\documentclass[runningheads]{llncs}

\usepackage{amssymb,amsmath}
\usepackage{color}
\usepackage{url}
\newcommand{\set}[1]{\left\{#1\right\}}                     
\newcommand{\abs}[1]{\left|#1\right|}                       
\newcommand{\bra}[1]{\left(#1\right)}  
\usepackage{stmaryrd}
\newcommand{\sidx}[1]{\left\llbracket     #1 \right\rrbracket}

\usepackage{tikz}
\definecolor{mathset}{HTML}{CFD9E8}
\definecolor{mathboundary}{HTML}{5E81B5}
\tikzstyle{notactive}=[draw=mathboundary,thick, fill=mathset]
\tikzstyle{notactiveinner}=[draw=mathboundary,thick,fill=white]
\DeclareMathOperator{\conv}{conv}
\DeclareMathOperator{\cconv}{\overline{conv}}

\DeclareMathOperator{\relint}{relint}
\DeclareMathOperator{\aff}{aff}
\DeclareMathOperator{\Int}{int}
\DeclareMathOperator{\dom}{dom}
\DeclareMathOperator{\proj}{proj}
\begin{document}
\mainmatter              
\title{Mixed-integer convex representability}
\titlerunning{Mixed-integer convex representability}  
%
\author{Miles Lubin$^*$ \and Ilias Zadik$^*$ \and Juan Pablo Vielma$^\dagger$}
\authorrunning{Lubin, Zadik, and Vielma} 
\institute{Massachusetts Institute of Technology, Cambridge, MA USA}

\maketitle              

\begin{abstract}
We consider the question of which nonconvex sets can be represented exactly as the feasible sets of mixed-integer convex optimization problems. We state the first complete characterization for the case when the number of possible integer assignments is finite. We develop a characterization for the more general case of unbounded integer variables together with a simple necessary condition for representability which we use to prove the first known negative results. Finally, we study representability of subsets of the natural numbers, developing insight towards a more complete understanding of what modeling power can be gained by using convex sets instead of polyhedral sets; the latter case has been completely characterized in the context of mixed-integer linear optimization.
\end{abstract}
\section{Introduction}
\makeatletter{\renewcommand*{\@makefnmark}{}
\footnotetext{$^*$ These authors contributed equally to this work.\\$^\dagger$ Supported by NSF under grant CMMI-1351619.\\We acknowledge the anonymous referees for improving the presentation of this work. }
\makeatother}

Early advances in solution techniques for mixed-integer linear programming (MILP) motivated studies by Jeroslow and Lowe~\cite{springerlink:10.1007/BFb0121015} and others (recently reviewed in~\cite{VielmaSurvey}) on understanding precisely which sets can be encoded as projections of points within a closed polyhedron satisfying integrality restrictions on a subset of the variables. These sets can serve as feasible sets in mixed-integer linear optimization problems and therefore potentially be optimized over in practice by using branch-and-bound techniques (ignoring issues of computational complexity). Jeroslow and Lowe, for example, proved that the epigraph of the piecewise linear function $f(x)$ which equals 1 if $x > 0$ and 0 if $x = 0$, is not representable over the domain $x \ge 0$. Such a function would naturally be used to model a fixed cost in production. It is now well known that an upper bound on $x$ is required in order to encode such fixed costs in an MILP formulation.

Motivated by recent developments in methods for solving mixed-integer convex programming (MICP) problems~\cite{BonamiReview,IPCO2016}, in this work we address the analogous question of which \textit{nonconvex} sets may be represented as projections of points within a \textit{convex} set satisfying integrality restrictions on a subset of the variables.   To our knowledge, we are the first authors to consider this general case. Related but more specific analysis has been developed by Del Pia and Poskin~\cite{del2016ellipsoidal} where they characterized the case where the convex set is an intersection of a polyhedron with an ellipsoidal region and by Dey and Mor{\'a}n~\cite{dey2013some} where they studied the structure of integer points within convex sets but without allowing a mix of continuous and discrete variables.

After a brief study in Section~\ref{sec:bounded} of restricted cases, e.g., when there is a finite number of possible integer assignments, we focus primarily on the more challenging general case where we seek to understand the structure of countably infinite unions of slices of convex sets induced by mixed-integer constraints. In Section~\ref{sec:family} we develop a general, yet hard to verify, characterization of representable sets as families of convex sets with specific properties, and in Section~\ref{sec:midpoint} we prove a much simpler necessary condition for representability which enables us to state a number of nonrepresentability results. Using that condition, we prove, for example, that the set of $m \times n$ matrices with rank at most 1 is not representable when $m,n \ge 2$. In Section~\ref{sec:naturals} we conclude with an in-depth study of the representability of subsets of the natural numbers. The special case of the natural numbers is a sufficiently challenging first step towards a general understanding of the structure of representable sets. We prove, for example, that the set of prime numbers is not representable, an interesting case that separates mixed-integer convex representability from mixed-integer polynomial representability~\cite{diophantine}. By adding rationality restrictions to the convex set in the MICP formulation, we completely characterize representability of subsets of natural numbers, discovering that one can represent little beyond what can be represented by using rational polyhedra.

\section{Preliminaries}

We use the notation $\sidx{k}$ to denote the set $\{1,2,\ldots,k\}$. Also by $\mathbb{N}$ we will refer to the nonnegative integers $\{0,1,2,\ldots\}$. We will often work with projections of a set $M\subseteq \mathbb{R}^{n+p+d}$ for some $n,p,d \in \mathbb{N}$. We identify the variables in $\mathbb{R}^n$, $\mathbb{R}^p$ and $\mathbb{R}^d$ of this set as $x$, $y$ and $z$ and we let 
\[\proj_x\bra{M}=\set{x\in \mathbb{R}^n\,:\,\exists \bra{y,z}\in \mathbb{R}^{p+d} \text{ s.t. } \bra{x,y,z}\in M}. \]
We similarly define $\proj_y\bra{M}$ and $\proj_z\bra{M}$.
\begin{definition}
   Let $M\subseteq \mathbb{R}^{n+p+d}$ be a closed, convex set and $S \subseteq \mathbb{R}^n$. We say $M$ induces an MICP formulation of $S$ if and only if 
   \begin{equation}\label{MICPdef}
   S=\proj_x\bra{M\cap \bra{\mathbb{R}^{n+p} \times \mathbb{Z}^d}}.
   \end{equation}

   A set $S \subseteq \mathbb{R}^n$ is \textbf{MICP representable} if and only if there exists an MICP formulation of $S$. If such formulation exists for a closed polyhedron $M$ then we say $S$ is (additionally) MILP representable. 
   \end{definition}
   
   \begin{definition}
    A set $S$ is \textbf{bounded MICP (MILP)} representable if there exists an MICP (MILP) formulation which satisfies $\abs{\proj_z\bra{M\cap \bra{\mathbb{R}^{n+p} \times \mathbb{Z}^d}}}<\infty$.
   That is, there are only finitely many feasible assignments of the integer variables $z$. 
   \end{definition}

   \begin{definition}
       For a set of integral vectors $z_1, z_2, \ldots, z_k \in \mathbb{Z}^d$ we define the integral cone $
       \operatorname{intcone}(z_1,z_2,\ldots,z_k):=\{ \sum_{i=1}^{k} \lambda_i z_i \big{|} \lambda_i \in \mathbb{N}, i \in \sidx{k}\}$.
   \end{definition}

\subsection{Bounded and other restricted MICP representability results}\label{sec:bounded}

It is easy to see that bounded MICP formulations can represent \textit{at most} a finite union of projections of closed, convex sets. To date, however, there are no precise necessary conditions over these sets for the existence of a bounded MICP formulation. For instance, Ceria and Soares~\cite{CeriaSoares} provide an MICP formulation for the finite union of closed, convex sets under the condition that the sets have the same \textit{recession cone} (set of unbounded directions). In the following proposition we close this gap and give a simple, explicit formulation for any finite union of projections of closed, convex sets without assumptions on recession directions.

\begin{proposition}\label{balasform}
$S\subseteq  \mathbb{R}^{n}$ is bounded MICP representable if and only if there exist nonempty, closed, convex sets $T_1, T_2, \ldots, T_k \subset \mathbb{R}^{n+p}$ for some $p,k \in \mathbb{N}$  such that $S=\bigcup_{i=1}^k \operatorname{proj}_{x} T_i$. In particular a formulation for such $S$ is given by 
\begin{subequations}\label{form:extendedbounded}
\begin{align}
x = \sum\nolimits_{i=1}^k x_i,\quad (x_i,y_i,z_i) \in \hat T_i \; \forall  i\in \sidx{k},\quad \sum\nolimits_{i=1}^k z_i = 1,\; z \in \{0,1\}^k,\\||x_i||_2^2 \le z_i t,\quad\; \forall i\in \sidx{k},  t \ge 0
\end{align}
\end{subequations}
where $\hat T_i$ is the closed conic hull of $T_i$, i.e., the closure of $\{ (x,y,z) : (x,y)/z \in T_i, z > 0 \}$. 
\end{proposition}

Known MICP representability results for unbounded integers are more limited. For the case in which $M$ is a rational polyhedron Jerowslow and Lowe \cite{springerlink:10.1007/BFb0121015} showed that a set $S\subseteq  \mathbb{R}^{n}$ is (unbounded) rational MILP representable if and only if there exist $r_1,r_2,\ldots,r_t \subseteq \mathbb{Z}^n$ and rational polytopes $S^i$ for $i\in \sidx{k}$ such that 
\begin{equation}\label{milprep}
S= \bigcup\nolimits_{i=1}^k S^{i} + \operatorname{intcone}(r_1,r_2,\ldots,r_t).
\end{equation}
Characterization \eqref{milprep} does not hold in general for non-polyhedral $M$. However, using results from  \cite{dey2013some} it is possible to show that it holds for some pure integer cases as well. For instance, Theorem 6 in \cite{dey2013some} can be used to show that for any $\alpha>0$, $S_{\alpha}:=\set{x\in \mathbb{Z}^2\,:\, x_1 x_2 \geq \alpha}$ satisfies \eqref{milprep} with $S^i$ containing a single integer vector for each $i\in \sidx{k}$.

The only mixed-integer and non-polyhedral result we are aware of is a characterization of the form \eqref{milprep} when $M$ is the intersection of a rational polyhedron with an  ellipsoidal cylinder having a rational recession cone \cite{del2016ellipsoidal}. An identical proof also holds when the recession cone of $M$ is a rational subspace and $M$ is contained in a rational polyhedron with the same recession cone as $M$. We can further extend this result to the following simple proposition whose proof is in the appendix.

\begin{proposition}\label{boundedprop}
If $M$ induces an MICP-formulation of $S$ and $M=C+K$ where $C$ is a compact convex set and $K$ is a rational polyhedral cone, then $S$ satisfies representation \eqref{milprep} with $S^{i}$ now being compact convex sets for each $i\in \sidx{k}$.
    \end{proposition}

Unfortunately, MICP-representable sets in general may not have a representation of the form  \eqref{milprep}, even when $S^i$ is allowed to be any convex set. We illustrate this with a simple variation on the pure-integer example above.
\begin{example}\label{ex:hyperbola}
Let $S:=\set{x\in \mathbb{N}\times \mathbb{R}\,:\, x_1 x_2 \geq 1}$ be the set depicted in Figure \ref{fig:pictures}. For each $z\in \mathbb{N}, z \ne 0$ let $A_z:=\set{x\in \mathbb{R}^2\,:\, x_1=z,\;x_2\geq 1/z}$ so that $S=\bigcup_{z=1}^\infty A_z$. Suppose for contradiction that $S$ satisfies  \eqref{milprep} for convex sets $S^i$. By convexity of $S^i$ and finiteness of $k$ there exists $z_0\in \mathbb{Z}$ such that $ \bigcup_{i=1}^k S^{i}\subset  \bigcup_{z=1}^{z_0-1} A_z$. Because $\min_{x\in A_{z_0}}x_2<\min_{x\in A_{z}}x_2$ for all $z\in \sidx{z_0-1}$ we have that there exists $j\in \sidx{t}$ such that the second component of $r_j$ is strictly negative. However, this implies that there exists $x\in S$ such that $x_2<0$ which is a contradiction with the definition of $S$.
\end{example}

\section{A general characterization of MICP representability}\label{sec:family}
The failure of characterizations of the form~\eqref{milprep} to hold calls for a more general characterization of MICP-representable sets as projections of families of sets with particular structure. Example~\ref{ex:hyperbola} hints at the union of a countable number of convex sets indexed by a set of integers. The following definition shows the precise conditions on this sets and indexes for the existence of a MICP formulation.
\begin{definition}
    Let $C\subseteq  \mathbb{R}^d$ be a convex set and $(A_z)_{z \in C}$ be a family of convex sets in $\mathbb{R}^n$. We say that the family of sets is \textbf{convex} if for all $z,z'  \in C$ and $\lambda \in [0,1]$ it holds $ \lambda A_z+ (1-\lambda) A_{z'} \subseteq A_{\lambda z+ (1-\lambda) z'}$. 

    We further say that the family is \textbf{closed} if $A_z$ is closed for all $z\in C$ and for any convergent sequences $\set{z_m}_{m\in \mathbb{N}},\set{x_m}_{m\in \mathbb{N}}$ with $z_m \in C$ and $x_m \in A_{z_m}$ we have $\lim_{m\to\infty}x_m \in A_{\lim_{m\to\infty}z_m}$.

\end{definition}

\begin{lemma}
Let $(A_z)_{z \in C}$ be a convex family and $C'\subseteq C$ be a convex set. Then $\bra{\proj \bra{A_z}}_{z\in C'}$ is a convex family, where $\proj$ is any projection onto a subset of the variables.
\end{lemma}
The proof of the above lemma is simple and it is omitted. 
\begin{theorem}\label{reptheo}
    A set $S \subseteq \mathbb{R}^n$ is MICP representable if and only if there exists $d \in \mathbb{N}$, a convex set $C\subseteq \mathbb{R}^d$ and a closed convex family $(B_z)_{z \in C}$ in $\mathbb{R}^{n+p}$ such that  $S=\bigcup_{z \in C \cap \mathbb{Z}^d} \proj_x\bra{B_z}$.
\end{theorem}

\begin{proof}
Suppose that $S$ is MICP representable. Then there exists $p,d \in \mathbb{N}$ and a closed and convex set $M \subseteq \mathbb{R}^{n+p+d}$ satisfying \eqref{MICPdef}. Let $C=\proj_z\bra{M}$ and for any $z \in C$ let $B_{z}=\set{\bra{x,y}\in \mathbb{R}^{n+p}\,:\,  \bra{x,y,z}\in M}$. The result follows by noting that    $(B_z)_{z \in C}$ is a closed convex family because $M$ is closed and convex.

For the converse, let 
 $M:=\cconv\left( \bigcup_{z \in C\cap \mathbb{Z}^d} B_z \times \{z \} \right)$. Set $M$ is closed and convex by construction and hence the only thing that remains to prove is that $B_z=\set{\bra{x,y}\,:\, \bra{x,y,z}\in M}$ for all $z\in C \cap \mathbb{Z}^d$.  The left to right containment is direct. For the reverse containment let $M':=\conv\left( \bigcup_{z \in C\cap \mathbb{Z}^d} B_z \times \{z \} \right)$ so that $M=\overline{M'}$ and $B'_z=\set{\bra{x,y}\,:\, \bra{x,y,z}\in M'}$ for all $z\in C$. Because $(B_z)_{z \in C}$ is a convex family we have $B'_z\subseteq B_z$ for all $z\in C$. Let $z\in C\cap \mathbb{Z}^d$ and $\set{\bra{x_m,y_m,z_m}}_{m\in\mathbb{N}}\subseteq M'$ be a convergent sequence such that $\lim_{m\to\infty}z_m=z$. We have for all $m$, $(x_m,y_m)\subseteq B'_{z_m}\subseteq B_{z_m}$ so $\lim_{m\to\infty}(x_m,y_m)\in B_z$ because $(B_z)_{z \in C}$ is a closed convex family.

\end{proof}
\begin{definition}
    For an MICP representable set $S \subseteq \mathbb{R}^n$ we let its MICP-dimension be the smallest $d'\in \mathbb{N}$ such that the representation from Theorem~\ref{reptheo} holds with $\dim\bra{C}=d'$.
\end{definition}

\begin{remark}
If $S=\bigcup_{z \in C \cap \mathbb{Z}^d} \proj_x\bra{B_z}$ for a convex set $C\subseteq \mathbb{R}^d$ and a closed convex family $(B_z)_{z \in C}$ in $\mathbb{R}^{n+p}$ and $C'=\conv\bra{C\cap \mathbb{Z}^d}$, then  $(B_z)_{z \in C'}$ is a  closed convex family and $S=\bigcup_{z \in C' \cap \mathbb{Z}^d} \proj_x\bra{B_z}$.
\end{remark}

\begin{remark}
Using the convex family characterization it can be proven that sets like the union of expanding circles with concave radii or the set described in Example~\ref{ex:hyperbola} are MICP representable; see Figure~\ref{fig:pictures}.
\end{remark}

\section{A necessary condition for MICP representability}\label{sec:midpoint}

In this section we prove an easy to state, and usually also to check, necessary property for any MICP representable set. Intuitively, it is saying that despite the fact that MICP representable sets could be nonconvex, they will never be very nonconvex in an appropriately defined way.

\begin{definition}
    We say that a set $S \subseteq \mathbb{R}^n$ is \textbf{strongly nonconvex}, if there exists a subset $R \subseteq S$ with $|R| = \infty$ such that for all pairs $x,y \in R$,
\begin{equation}
\frac{x+y}{2} \not \in S,
\end{equation}
that is, an infinitely large subset of points in $S$ such that the midpoint between any pair is not in $S$.
\end{definition}
\begin{lemma} \textbf{(The midpoint lemma)} \label{midpoint}
Let $S \subseteq \mathbb{R}^n$. If $S$ is strongly nonconvex, then $S$ is not MICP representable.
\end{lemma}
\begin{proof}
Suppose we have $R$ as in the statement above and there exists an MICP formulation of $S$, that is, a closed convex set $M \subset \mathbb{R}^{n+p+d}$ such that $x \in S$ iff $\exists z \in \mathbb{Z}^d,y \in \mathbb{R}^p$ such that $(x,y,z) \in M$. 
Then for each point $x \in R$ we associate at least one integer point $z_x \in \mathbb{Z}^d$ and a $y_x \in \mathbb{R}^p$ such that $(x,y_x,z_x) \in M$. If there are multiple such pairs of points $z_x,y_x$ then for the purposes of the argument we may choose one arbitrarily.

We will derive a contradiction by proving that there exist two points $x,x' \in R$ such that the associated integer points $z_x,z_{x'}$ satisfy
\begin{equation}\label{eq:vecparity}
\frac{z_x+z_{x'}}{2} \in \mathbb{Z}^d.
\end{equation}
Indeed, this property combined with convexity of $M$, i.e., $\left(\frac{x+x'}{2},\frac{y_x+y_{x'}}{2},\frac{z_x+z_{x'}}{2}\right) \in M$
would imply that $\frac{x+x'}{2} \in S$, which contradicts the definition of $R$.

Recall a basic property of integers that if $i,j \in \mathbb{Z}$ and $i \equiv j \text{ (mod 2)}$, i.e., $i$ and $j$ are both even or odd, then $\frac{i+j}{2} \in \mathbb{Z}$. We say that two integer vectors $\alpha,\beta \in \mathbb{Z}^d$ have the same \textit{parity} if $\alpha_i$ and $\beta_i$ are both even or odd for each component $i = 1,\ldots,d$. Trivially, if $\alpha$ and $\beta$ have the same parity, then $\frac{\alpha+\beta}{2} \in \mathbb{Z}^d$. Given that we can categorize any integer vector according to the $2^d$ possible choices for whether its components are even or odd, and we notice that from any infinite collection of integer vectors we must have at least one pair that has the same parity. Therefore since $|R| = \infty$ we can find a pair $x,x' \in R$ such that their associated integer points $z_x,z_{x'}$ have the same parity and thus satisfy~\eqref{eq:vecparity}.\qed
\end{proof}

\begin{figure}[t]
    \centering
\begin{tikzpicture}[scale=0.47]
\path[draw=black,use as bounding box] (-3,-3) rectangle (3,3);

\draw[notactive] (0,0) circle (2cm);
\draw[notactiveinner] (0,0) circle (1.5cm);

\end{tikzpicture}
\begin{tikzpicture}[scale=0.47]
\path[draw=black,use as bounding box] (-3,-0.5) rectangle (3,5.5);

\node[circle,fill=mathboundary,scale=0.5] at (0,0) {};
\node[circle,fill=mathboundary,scale=0.5] at (1,1) {};
\node[circle,fill=mathboundary,scale=0.5] at (2,4) {};
\node[circle,fill=mathboundary,scale=0.5] at (-1,1) {};
\node[circle,fill=mathboundary,scale=0.5] at (-2,4) {};
\draw[notactive,fill=none] (-2,4) -- (-1,1) -- (0,0) -- (1,1) -- (2,4);
\node[draw=none,fill=none,rotate=80] at (2.2,4.9) {$\cdots$};
\node[draw=none,fill=none,rotate=-80] at (-2.2,4.8) {$\cdots$};

\end{tikzpicture}
\begin{tikzpicture}[scale=0.47]
\path[draw=black,use as bounding box] (0,0) rectangle (6,6);

\draw[notactive,fill=none] (1,1) -- (1,6);
\draw[notactive,fill=none] (2,1/2) -- (2,6);
\draw[notactive,fill=none] (3,1/3) -- (3,6);
\draw[notactive,fill=none] (4,1/4) -- (4,6);
\node[draw=none,fill=none] at (5,1) {$\cdots$};
\draw[dashed,smooth,samples=100,domain=0.1666:6] plot(\x,{(1/\x)});

\end{tikzpicture}
\begin{tikzpicture}[scale=0.47]
\path[draw=black,use as bounding box] (0,-3) rectangle (6,3);

\draw[notactive] (1,0) circle (0.8*0.6);
\draw[notactive] (2.5,0) circle (0.8*0.882);
\draw[notactive] (4,0) circle (0.8*0.969);

\node[draw=none,fill=none] at (5.5,0) {$\cdots$};
\draw[dashed,smooth,samples=100,domain=0:6] plot(\x,{0.8*(2^((2*\x+1)/3)-2^(-((2*\x+1)/3)) )/ (2^((2*\x+1)/3)+2^(-((2*\x+1)/3)))});
\draw[dashed,smooth,samples=100,domain=0:6] plot(\x,{-0.8*(2^((2*\x+1)/3)-2^(-((2*\x+1)/3)) )/ (2^((2*\x+1)/3)+2^(-((2*\x+1)/3)))});

\end{tikzpicture}

    \caption{From left to right, the annulus and the piece-wise linear function connecting the integer points on the parabola \textit{are not} mixed-integer convex representable. The mixed-integer hyperbola and the collection of balls with increasing and concave radius \textit{are} mixed-integer convex representable.}\label{fig:pictures}
\end{figure}
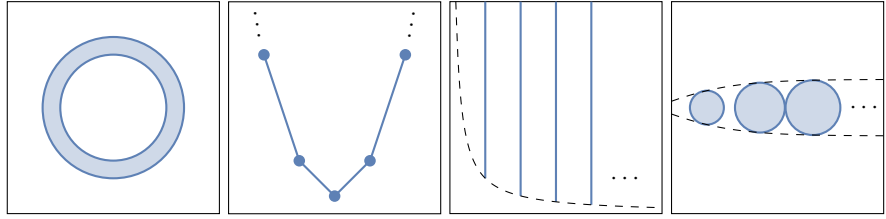

\begin{proposition}
Fix $n,m \in \mathbb{N}$ with $m,n \geq 2$. The set of matrices of dimension $m \times n$ with rank at most $1$, i.e., $C_1 := \{ X \in \mathbb{R}^{m\times n} : \operatorname{rank}(X) \le 1 \}$ is strongly nonconvex and therefore not MICP representable.

\end{proposition}

\begin{proof}
We can assume $m=2$. We set for all $k \in \mathbb{N}$ the matrix 
$A_k = \begin{bmatrix}
    1      & k   & O_{1 \times n-2}\\
    k       & k^2 & O_{1 \times n-2}
\end{bmatrix} \in C_1$. We then set  $R=\{A_k | k \in \mathbb{N}  \}$. Clearly $|R|=\infty$. It is easy to verify that $\operatorname{rank}(\frac{1}{2} (A_k+A_{k'})) = 2$ for $k \ne k'$.
Therefore for any pair of distinct points in $R$, their midpoint is not in $C_1$. Therefore $C_1$ is strongly nonconvex and in particular not MICP representable.

\end{proof}
One may use the midpoint lemma to verify that the epigraph of a twice differentiable function is MICP representable if and only if the function is convex and that the graph of a twice differentiable function is MICP representable if and only if $f$ is linear. In Figure~\ref{fig:pictures}, we illustrate two more sets whose nonrepresentability follows directly from the midpoint lemma: the annulus and the piecewise linear function connecting the points $\{ (x,y) \in \mathbb{Z}^2 : y = x^2 \}$.

\section{MICP representability of subsets of natural numbers}\label{sec:naturals}

Recall that we define $\mathbb{N}=\{0,1,2,\ldots \}$ to be the set of natural numbers.
In this section we investigate the limitations of MICP for representing subsets of the natural numbers. We remind the reader that in the MILP case, it is known that a subset of the natural numbers is rational-MILP (the coefficients of the continuous relaxation polyhedron are rational numbers) representable if and only if the set is equal to the Minkowski summation of finitely many natural numbers plus the set of nonnegative integer combinations of a finite set of integer generators~\cite{VielmaSurvey}. We simplify this characterization since we are dealing with subsets of the natural numbers.

We define an infinite arithmetic progression in the natural numbers to be a sequence of natural numbers of the form $am+b,m  \in \mathbb{N}$  for some fixed $a,b \in \mathbb{N}$.

 \begin{lemma}\label{MILPlemma}
Let $S \subseteq \mathbb{N}$. $S$ is rational-MILP representable if and only if $S$ is the union of finitely many infinite arithmetic progressions with the same nonnegative step size.
\end{lemma}   

The proof of Lemma \ref{MILPlemma} is in the appendix.
We now compare MICP representability with rational-MILP representability on $\mathbb{N}$.  
To be able to deduce a characterization for MICP in the naturals similar to the one we have for rational-MILP in $\mathbb{N}$ it is natural to put some ``rationality" restrictions on the MICP representations as well. 

For instance we could require $M$ to have a representation of the form $Ax-b\in K$ where $A$ and $b$ are an  appropriately sized rational matrix and rational vector, and $K$ is a specially structured convex cone (e.g. the semidefinite cone or a product of Lorentz cones defined as $\mathcal{L}_n := \{ (t,x) \in \mathbb{R}^n : ||x||_2 \le t\}$) or to have polynomial constraints with rational coefficients. Unfortunately these restrictions can still result in representable sets that do not contain any infinite arithmetic progression and hence are far from being rational-MILP representable. We present such an example below.

\begin{example}\label{ex:nonrational}
For $x\in \mathbb{R}$ let $f(x)=x-\lfloor x \rfloor$. For $\varepsilon>0$ consider the set 
\begin{align}
K_\varepsilon&=\{x\in \mathbb{R}^2\,:\, (x_2+\varepsilon,x_1,x_1) \in \mathcal{L}_3,(2x_1+2\varepsilon,x_2,x_2) \in \mathcal{L}_3,x_1,x_2\geq 0\}\\
     &=\{x\in \mathbb{R}^2\,:\, \sqrt{2} x_1 -  \varepsilon \leq x_2 \leq  \sqrt{2} x_1 + \sqrt{2} \varepsilon,\quad x_1,x_2\geq 0\}
\end{align}
and 
$S_\varepsilon=\{x_1\in \mathbb{R}: \exists x_1 \text{ s.t. } (x_1,x_2)\in K_\varepsilon \cap \mathbb{Z}^2\}=\{x\in \mathbb{N}: f(\sqrt{2}x )\notin (\varepsilon, 1-\sqrt{2}\varepsilon) \}$. Let $\varepsilon_0 <1/(1+\sqrt{2})$ be rational (e.g. $\varepsilon=0.4$). Suppose that for some $a,b \in \mathbb{N}$,$a \geq 1$ it holds  $a k +b \in S_{\varepsilon_0}$ for all $k\in \mathbb{N}$. $\emptyset\neq (\varepsilon_0, 1-\sqrt{2}\varepsilon_0)\subseteq (0,1)$, so by Kronecker’s Approximation Theorem we have that there exist $k_0\in \mathbb{N}$ such that $f(\sqrt{2}(a k_0 +b) )\in (\varepsilon_0, 1-\sqrt{2}\varepsilon_0)$ which is a contradiction. Therefore the set $S_\varepsilon$ does not contain an arithmetic progression and in particular it is is not rational-MILP representable.
\end{example}

We follow now a different path to define what rational MICP-representability is and we characterize it completely. Quite surprisingly it becomes almost equivalent with rational-MILP representability.

We give the following definitions
\begin{definition}\label{rationaldefc}
    We say that an unbounded convex set $C \subseteq \mathbb{R}^d$ is \textbf{rationally unbounded} if the image $C'$ of any rational linear mapping of $C$, is either bounded or there exists $r\in \mathbb{Z}^d\setminus \set{0}$ such that $x+\lambda r \in C'$ from any $x \in \mathbb{Z}^d \cap C'$ and $\lambda\geq 0$.

    Let $A \subseteq \mathbb{Z}^d$ be an infinite set of integer points. We say that $A$ is \textbf{rationally unbounded} if there exists a finite subset $I \subset A$ such that the set $\conv(A \setminus I)$ is rationally unbounded.

    Finally, we say that a set $S$ is \textbf{rational-MICP representable} if there exists an MICP representation for $S$ with convex family $(A_x)_{x \in C}$ with $A_x \ne \emptyset, \forall x \in C$ such that the set of integer points $C \cap \mathbb{Z}^d$ is either bounded or rationally unbounded.
    
\end{definition}
It is easy to see that the set $K_{\epsilon}$ from Example \ref{ex:nonrational} is not rationally unbounded.

To completely characterize the rational-MICP representable subsets of the natural numbers we will use the following lemmata. We prove Lemma \ref{lem:finiteunion} in the appendix and Lemma \ref{technicalLemma} in Section 6.

\begin{lemma}\label{lem:finiteunion}
Any union of finitely many infinite arithmetic progressions is equal to a union of finitely many arithmetic progressions with the same step length.
\end{lemma}

\begin{lemma}\label{technicalLemma}
Suppose $S \subseteq \mathbb{N}$ is rational MICP representable with MICP-dimension $d'$. 
Then either $S$ is a finite set or there exists $k \in \mathbb{N}$ such that $S=T_0 \cup S_0 \cup \bigcup_{i=1}^k S_i$ where $T_0$ is a finite set of natural numbers, $S_0$ is a finite union of infinite arithmetic progressions and for each $i\in \sidx{k}$, $S_i$ is rational-MICP representable with MICP-dimension at most $d'-1$.
\end{lemma}

\begin{theorem}
Suppose that $S \subseteq \mathbb{N}$.  Then following are equivalent.
\begin{itemize}
    \item[(a)] $S$ is rational-MICP representable 
\item[(b)] There exists $k \in \mathbb{N}$ such that $S=A_0 \cup \left( \bigcup_{i=1}^k A_i \right)$, where $A_0 \subseteq \mathbb{N}$ is a finite set and for each $i=1,\ldots,k,$ $A_i \subseteq \mathbb{N}$ is an infinite arithmetic progression. 
\item[(c)] There exists a finite set $A_0$ and a rational-MILP representable set $T$ such that $S=A_0 \cup T$. 
\end{itemize}
\end{theorem}
\begin{proof}
We start by proving that $(c)$ implies $(a)$. We will use Lemma \ref{MILPlemma}. Say $A_0 = \{a_1,\ldots,a_m\}$ and $T = \{ b_1,\ldots,b_n \} + \operatorname{intcone}(z)$. Then $x \in A_0 \cup T$ iff $\exists\, x_1,x_2,\beta,q,\alpha,\nu,\lambda,\eta,t$ such that
\begin{alignat*}{3}
&x = x_1 + x_2,\quad x_1 = \sum\nolimits_{i=1}^n \beta_i + qz,\quad x_2 = \sum\nolimits_{i=1}^m \alpha_i,\quad \beta_i = b_i\nu_i \; \forall  i \in \sidx{n},\\ 
&\alpha_i = a_i\lambda_i \; \forall i \in \sidx{m},\quad \sum\nolimits_{i=1}^n \nu_i = \eta,\quad  \sum\nolimits_{i=i}^m \lambda_i = 1-\eta,\quad q^2 \le \eta t, \\ &t \ge 0,\quad \lambda_i \in \{0,1\}\; \forall  i \in \sidx{m},\quad \nu_i \in \{0,1\}\;\forall i \in\sidx{n},\quad \eta \in \{0,1\}, q \in \mathbb{N}
\end{alignat*}
We claim that this is a rational-MICP representation.
Consider the integer variables $\lambda_i,\nu_i,\eta,q$. By excluding the finitely many integer extreme points with $\eta=0$ we have that it is enough to consider only the integer points with $\eta=1$ which imply all $\lambda_i=0$ and hence we have to consider only the integer variables $\nu_i,q$ that satisfy $\sum_{i=1}^n \nu_i=1, \nu_i \in \{0,1\},i=1,2,\ldots,n$ and $q \in \mathbb{N}$. The convex hull of these integer points is $\mathbb{R}_{+} \times \{ x \in \mathbb{R}_{+}^n | \sum_{i=1}^n x_i=1 \}$. But this is rationally unbounded, as it has exactly one rational recession direction $e_1:=(1,0,0,\ldots,0)$ and any rational linear map $t$ will either satisfy $t(e_1) \not = 0$, in which case $t(e_1)$ is a rational recession direction for the image, or $t(e_1)=0$, in which case the image is bounded.

Now $(b)$ implies $(c)$ because of Lemma \ref{lem:finiteunion} and noticing that finite union of infinite arithmetic progressions with the same step length is immediately rational-MILP representable because of Lemma \ref{MILPlemma}.

Finally we prove that $(a)$ implies $(b)$.  Suppose $S$ is rational-MICP representable.   We will use Lemma \ref{technicalLemma}. We first apply to $S$. If it is finite we are done. If not we apply it to each of the $S_i$, $i=1,2,\ldots, m$ produced by Lemma \ref{technicalLemma}. Continuing like this with at most $d$ iterations we prove our result.

\end{proof}

Despite the similarity that the above result indicates, rational MICP-representable subsets of the natural numbers and rational MILP-representable subsets of the natural numbers are not identical as the example below illustrates.

\begin{example}
Consider the set $S=\{1\} \cup 2 \mathbb{N}$. Then the set is rational-MICP representable from the above theorem. On the other hand, it cannot be written as the Minkowski summation of a finite set plus a finitely generated integral monoid and therefore it is not rational-MILP representable. To see the last, suppose it could be written like this by contradiction. Then consider one the generators of the monoid $z_1$. Assume $z_1$ is odd. Then $2+z_1$ should belong to $S$ but it is odd and bigger than 1, a contradiction. Assume $z_1$ is even. Then $1+z_1$ should belong to $S$ but it is odd and bigger than 1, a contradiction. The proof is complete.
\end{example}

We end the section with a global limitation of MICP representability in the subsets of the integers which hold without any type of rationality restriction. Its proof is based on the midpoint lemma.

\begin{theorem}
    The set of prime numbers $\mathbb{P}$ is strongly nonconvex and therefore not MICP representable.
    
    \end{theorem}
    \begin{proof}
    We will inductively construct a subset of primes such that no midpoint of any two elements in the set is prime.
    
Let $\{p_1,\ldots,p_n\}$ be a set of such primes. We will find a prime $p$ such that $\{p_1,\ldots,p_n,p\}$ has no prime midpoints. (We may start the induction with $p_1 = 3$, $p_2=5$.)

Set $M = \prod_{i=1}^n p_i$. Choose any prime $p$ (not already in our set and not equal to 2) such that $p \equiv 1 \pmod{M!}$. By Dirichlet's theorem on arithmetic progressions there exist an infinite number of primes of the form $1 + kM!$ because 1 and $M!$ are coprime, so we can always find such $p$.

Suppose for some $i$ we have $q := \frac{p+p_i}{2} \in \mathbb{P}$. By construction, we have $p + p_i \equiv 1 + p_i\pmod{M!}$, so $\exists\, k$ such that $p + p_i = k\cdot M! + 1 + p_i$. Note that $M$ is larger than $p_i$ so $M!$ will contain $(1+p_i)$ as a factor; in other words, $(1+p_i)$ divides $M!$, so it divides also $k \cdot M!+1+p_i=p+p_i$. In fact we can write $p+p_i = k'(1+p_i)$ for some $k' \in \mathbb{Z}_{\geq 0}$. We claim that $k'=1$. Indeed $q =\frac{p+p_i}{2}= k'\frac{1+p_i}{2}$. Note $1+p_i$ is even, so $\frac{1+p_i}{2}$ is an integer bigger than 1 as $p_i>1$. But $q$ is prime and therefore since it is written as the product of $k'$ and $\frac{1+p_i}{2}>1$ it must be the case that $k'=1$ as claimed. But $k'=1$ implies that $p+p_i =1+p_i$, i.e., $p = 1$ which is a contradiction.
    \end{proof}

\section{Proof of Lemma~\ref{technicalLemma}}

We first state Lemma~\ref{linearlemma} whose proof is given in the appendix.
\begin{lemma}\label{linearlemma}
Let $C\subseteq \mathbb{R}^d$ be a convex set, $h:C\to \mathbb{R}$ a nonpositive convex function and $\set{x^i}_{i=1}^k\subset C$ such that $h(x^1)= 0$ and $x^1\in \relint\bra{\aff\bra{\set{x^i}_{i=1}^k}\cap C}$. Then $h(x)=0$ for all $x\in \aff\bra{\set{x^i}_{i=1}^k}\cap C$.
\end{lemma}
\begin{proof}[of Lemma \ref{technicalLemma}]
Let $C\subseteq \mathbb{R}^d$ be the convex set such that $\dim\bra{C}=d'$ and $\set{B_z}_{z\in C}$ be the closed convex family such that $S=\bigcup_{z\in C\cap \mathbb{Z}^d} \proj_x\bra{B_z}$. Since $S$ is rational-MICP representable, $C\cap \mathbb{Z}^d$ is either finite or rationally unbounded. In the first case $S$ is finite so we can assume that $C\cap \mathbb{Z}^d$ is rationally unbounded. 

Since $n=1$ and convex subsets of the real line are intervals, we may define $f,g:C\to \mathbb{R}$ such that $f(z)$ and $g(z)$ represent the lower and upper endpoints of the intervals $\proj_x\bra{B_z}$ for all $z\in C$. Because $\set{\proj_x\bra{B_z}}_{z\in C}$ is a convex family we  have that  $h:=f-g:C \rightarrow \mathbb{R}$ is a convex function and $h(z)\leq 0$ for all $z\in C$. Furthermore, since $S\subseteq \mathbb{N}$ we have $h(z)=0$ for all $z\in C\cap \mathbb{Z}^d$. 

Let $I\subseteq C\cap \mathbb{Z}^d$ be the finite set such that $C'=\conv\bra{\bra{C\cap \mathbb{Z}^d}\setminus I}$ is rationally unbounded. By letting $T_0:=\set{\proj_x\bra{B_z}}_{z\in I}\subset \mathbb{N}$ be the finite set in the statement of Lemma \ref{technicalLemma} and noting that $C'\subseteq C$ we may redefine $C$ to be equal to $C'$. Let $r\in \mathbb{Z}^d$ be the direction from Definition~\ref{rationaldefc} such that $l_z:=\set{z+\lambda r\,:\,\lambda\geq 0} \subseteq  C$. Because $\abs{l_z\cap \mathbb{Z}}=\infty$ and $h(z')=0$ for all $z'\in l_z\cap \mathbb{Z}$ we have that $h(z')=0$ for all $z'\in l_z$ by Lemma~\ref{linearlemma}. Hence for all $z'\in l_z$ we have $\proj_x\bra{B_{z'}}=\set{f(z')}=\set{g(z')}$. Given that $\proj_x\bra{B_{z'}} \in \mathbb{N}$ for $z' \in l_z$ and being a convex and a concave function is equivalent to being an affine function, we further have that $\set{\proj_x\bra{B_z}}_{z\in l_z}$ being a convex family implies that there exist $\alpha_z \in \mathbb{Z}^d$ and $\beta_z \in \mathbb{Z}$ such that $f(z')=g(z')=\alpha_z \cdot z'+\beta_z$ for all $z'\in l_z$. Then  $\set{\proj_x\bra{B_{z'}}}_{z'\in l_z\cap \mathbb{Z}^d}=\set{a_z m +b_z\,:\, m\in \mathbb{N}_0 }$, where $a_z=(\alpha_z \cdot r)/\gcd\bra{r_1,\ldots,r_d}$ and $b_z=\beta_z$. If $\alpha_z \cdot r > 0$ this corresponds to an infinite arithmetic progression and if $\alpha_z \cdot r = 0$ it corresponds to a single point.

Let $\set{T_i}_{i=1}^{2^{d}}$ be such that $C\cap \mathbb{Z}^d=\bigcup_{i=1}^{2^{d}}T_i$ and  $z_j\equiv z'_j \mod 2$ for all $j\in \sidx{d}$,  $i\in \sidx{2^{d}}$ and $z,z'\in T_i$. For fixed $i\in \sidx{2^{d}}$ we have  $\frac{z+z'}{2} \in C\cap \mathbb{Z}^d$ and  $l_{\frac{z+z'}{2}}\subset C$ for any $z,z'\in T_i$. Then $P:=\conv\bra{\set{l_z,l_{\frac{z+z'}{2}},l_{z'}}}\subset C$ and $h\bra{\tilde{z}}=0$ for all $\tilde{z}\in P\cap \mathbb{Z}^d$. Then, by Lemma~\ref{linearlemma} $h\bra{\tilde{z}}=0$ for all $\tilde{z}\in P$. By the same argument in the previous paragraph there exist  $\alpha_P \in \mathbb{Z}^d$ and $\beta_P \in \mathbb{Z}$ such that $f(\tilde{z})=g(\tilde{z})=\alpha_P \cdot \tilde{z}+\beta_P$ for all $\tilde{z}\in P$. In particular, $\alpha_P \cdot \tilde{z}+\beta_P=\alpha_z \cdot \tilde{z}+\beta_z$ for all $\tilde{z}\in l_z$ and $\alpha_P \cdot \tilde{z}+\beta_P=\alpha_{z'} \cdot \tilde{z}+\beta_{z'}$ for all $\tilde{z}\in l_{z'}$. Hence $\alpha_{z} \cdot r=\alpha_{z'} \cdot r$ and then $s_i:=a_z=a_{z'}$. Then $\set{\proj_x\bra{B_{\tilde{z}}}}_{\tilde{z}\in l_z\cap \mathbb{Z}^d}=\set{s_i m +b_z\,:\, m\in \mathbb{N}_0 }$ for all $z\in T_i$. Unfixing $i$ we may define $S_0$ from the statement of Lemma~\ref{technicalLemma} to be
\[S_0:=\bigcup_{i\in \sidx{2^{d}}: s_i>0} \bigcup_{z\in T_i} \set{s_i m +b_z\,:\, m\in \mathbb{N}_0 }\subseteq \bigcup_{i\in \sidx{2^{d}}: s_i>0} \bigcup_{b=0}^{s_i-1} \set{s_i m +b\,:\, m\in \mathbb{N}_0 }.\] The last inclusion implies   $S_0$ is a finite union of infinite arithmetic progressions. It then only remains to consider sets $\set{\proj_x\bra{B_{\tilde{z}}}}_{\tilde{z}\in l_z\cap \mathbb{Z}^d}$ for $z\in T_i$ and $i\in \sidx{2^{d}}$ such that $s_i=0$ ($s_i\geq 0$ because $S\subseteq \mathbb{N}$). Say we have $k$ such $i$'s (WLOG $i=1,\ldots,k$) and we will show that for every such $i$, $S_i:=\bigcup_{z' \in T_i} \proj_x\bra{B_{{z'}}^i}$ is rational-MICP representable with MICP-dimension at most $d'-1$. Because $S=T_0 \cup S_0 \cup \bigcup_{i=1}^{k} S_i$ the proof will be complete.

For a fixed $i\in \sidx{2^{d}}$ such that $s_i=0$, let  $t^i \in T_i$ so that $T_i=\bra{t^i+2 \mathbb{Z}^d}\cap C$. Let $\set{v_i}_{i=1}^{d}$ a rational orthogonal basis of $\mathbb{R}^d$ such that $r=v_{d}$ and $\set{v_i}_{i=d-d'+1}^{d}$ is an orthonormal basis of the linear subspace $L(C)$ parallel to $\aff(C)$ (i.e. $L(C):=\aff\bra{C-z}$ for any $z\in C$).
 Let $A\in \mathbb{R}^{d\times d}$ such that for $i\leq d-1$ the $i$-th row of $A$ is $v_i^T$ and the $d$-th row of $A$ has all components equal to zero. Also, let ${A}_{1:d-1}$ be the restriction of $A$ to the first $d-1$ rows and let $H\in \mathbb{R}^{\bra{d-1}\times\bra{d-1}}$ and $U\in \mathbb{R}^{d\times d}$ be a unimodular matrix such that
 \begin{equation}\label{HNF}
 {A}_{1:d-1}=[H|0]U
 \end{equation}
 (e.g Hermite normal form). Finally, let  $l^{\pm}_z:=\set{z+\lambda r\,:\,\lambda\in \mathbb{R}}\cap C$ and
 $C_i= U^{-1} \begin{bmatrix}H^{-1}&0\\0&1\end{bmatrix} A (C-t^i)/2= U^{-1} \begin{bmatrix}I&0\\0&0\end{bmatrix} U (C-t^i)/2$. 

 We claim that $\bigcup_{w \in T_i } l^{\pm}_{w}=\bigcup_{w \in C_i \cap \mathbb{Z}^d } l^{\pm}_{2w+t^i}$. 
 Indeed,   $z'\in T_i$ if and only if there exists $z''\in \mathbb{Z}^d\cap \bra{C-t^i}/2$ such that $z'=t^i+ 2z''$. Then $ A z''\in  A (C-t^i)/2$ and $l^{\pm}_{z'}=l^{\pm}_{t^i+2z''}$. But we know $A=\begin{bmatrix}H&0\\0&0\end{bmatrix} U$ which gives after some algebra $  U z'' \in \begin{bmatrix}I&0\\0&0\end{bmatrix} U (C-t^i)/2 +\begin{bmatrix}0&0\\0&1\end{bmatrix} \mathbb{R}^d$. However, because $U$ is unimodular and $z''\in \mathbb{Z}^d$ we have  $U z''\in \mathbb{Z}^d$ we can replace $\mathbb{R}^d$ by $\mathbb{Z}^d$ and there exist $y \in \mathbb{Z}$ and $z \in C_i\cap \mathbb{Z}^d$ such that  $z''=z+U^{-1} \begin{bmatrix}0\\y\end{bmatrix}$. From \eqref{HNF}, orthogonality of  $\set{v_i}_{i=1}^{d}$ and unimodularity of $U$ we have $U^{-1} e_d =\alpha r$ for some $\alpha\in \mathbb{Z}$ and hence  $z''=z+y \alpha r$. Then $l^{\pm}_{z'}=l^{\pm}_{t^i+2z''}=l^{\pm}_{t+2z}\subseteq \bigcup_{w \in C_i \cap \mathbb{Z}^d } l^{\pm}_{2w+t^i}$.

 For the other direction, if $z\in C_i\cap \mathbb{Z}^d$ then  there exist $z''\in \bra{C-t^i}/2$ such that $U z = \begin{bmatrix}I&0\\0&0\end{bmatrix} U z''= U z''- e^d [U z'']_d= U \bra{z''-\alpha r [U z'']_d}$ for  the $\alpha\in \mathbb{Z}$ such that $U^{-1} e_d =\alpha r$. Let $z'''= z''-\alpha r [U z'']_d$ and $\mu=\alpha  [U z'']_d$. Then $z'''\in \mathbb{Z}^d$ and $z'''+\mu r\in \bra{C-t^i}/2$ and hence there exist $k\geq \mu$ such that $\bar{z}=z'''+k r\in \bra{C-t^i}/2\cap \mathbb{Z}^d$ (because without loss of generality we may replace $C$ with $\conv\bra{C\cap \mathbb{Z}^d}$ so that $r$ is a recession direction of $C$). Finally, $z= z''-\alpha r [U z'']_d=\bar{z}-k r$. Then $z':=t^i+2(z+k r)$ is such that  $z' \in T_i$ and $l^{\pm}_{t^i+2z}=l^{\pm}_{z'} \subseteq \bigcup_{w \in T_i } l^{\pm}_{w}$ as claimed.

Now  let   $l^{\pm}\bra{z,\lambda}:=z+\lambda r$ and $\Lambda:=\set{\lambda\in \mathbb{R}\,:\,l^{\pm}\bra{z,\lambda}\in C }$ so that $l^{\pm}_z=\set{z+\lambda r\,:\,\lambda\in \mathbb{R}}\cap C=\bigcup_{\lambda\in \Lambda} l^{\pm}\bra{z,\lambda}$ and $\Lambda$ is a convex set in $\mathbb{R}$. Furthermore, for each $z\in C_i$ let 
$\tilde{B}_z^i=\bigcup_{\lambda \in \Lambda} \left(B_{l^{\pm}\bra{t^i+2z,\lambda}} \times \{ \lambda \}\right)$. We can check that $\bra{\tilde{B}_z^i}_{z\in C_i}$ is a closed convex family. Both convexity of $\tilde{B}_z^i$ and the convex family property hold because $\bra{{B}_z}_{z\in C}$ is a convex family, $l^{\pm}_z$ is convex and $l^{\pm}\bra{z,\lambda}$ is affine. To see that the family is closed, suppose we have a sequence $(x^m,y^m,\lambda^m,z^m)$ converging to $(x,y,\lambda,z)$ with $(x^m,y^m,\lambda^m) \in \tilde{B}_{z^m}^i$. Then $(x^m,y^m)\in B_{l^{\pm}\bra{t^i+2z^m,\lambda^m}}$ and hence by closedness of $\bra{{B}_z}_{z\in C}$ and continuity of $l^{\pm}\bra{z,\lambda}$ we have $(x,y)\in B_{l^{\pm}\bra{t^i+2z,\lambda}}$ and hence $(x,y,\lambda) \in \tilde{B}_{z}^i$. Finally,  for each $z\in C_i\cap \mathbb{Z}^d$ we have 
\[\proj_x\bra{\tilde{B}_{{z}}^i}=\bigcup_{\lambda\in \Lambda} \proj_x\bra{B_{l^{\pm}\bra{t^i+2z,\lambda}}} =\bigcup_{\tilde{z}\in l^{\pm}_{t^i+2z}}\proj_x\bra{B_{\tilde{z}}}=\bigcup_{\tilde{z}\in l^{\pm}_{z'}}\proj_x\bra{B_{\tilde{z}}}\] for some $z' \in T_i$ where we set $B_y=\emptyset$ for all $y \not \in C$ and we have used  $ \bigcup_{w \in C_i \cap \mathbb{Z}^d } l^{\pm}_{2w+t^i} \subset \bigcup_{w \in T_i } l^{\pm}_{w}$.  
However, because $s_i=0$ we have that $\proj_x\bra{B_{\tilde{z}}}=\{b_{z'} \}$ for all  $\tilde{z}\in l^{\pm}_{z'} \cap C$. Hence for all $z \in C_i \cap \mathbb{Z}^d$, $\proj_x\bra{\tilde{B}_{{z}}^i}=\{b_{z'} \}=\proj_x\bra{B_{{z'}}}$ for some $z' \in T_i$. But for any $z' \in T_i$ since $\bigcup_{w \in T_i } l^{\pm}_{w} \subset \bigcup_{w \in C_i \cap \mathbb{Z}^d } l^{\pm}_{2w+t^i}$ it holds also  $\proj_x\bra{B_{{z'}}}=\proj_x\bra{\tilde{B}_{{z}}^i}$ for some $z \in C_i \cap \mathbb{Z}^d$. Hence $S_i=\bigcup_{z' \in T_i} \proj_x\bra{B_{{z'}}}=\bigcup_{z\in C_i\cap \mathbb{Z}^d}\proj_x\bra{\tilde{B}_z^i}$.
The result finally follows since $\tilde{B}_z^i$ is a closed convex family, by noting that $\dim\bra{C_i}\leq d'-1$, $C_i$ is rationally unbounded by definition of $C$ as a rational map of $C$ and hence $S_i$ is rational-MICP representable with MICP-dimension at most $d'-1$.

\end{proof}

\bibliographystyle{splncs03}
\bibliography{refs}

\pagebreak
\section*{Appendix}
\appendix
\section{Additional Proofs}

\subsection{Proof of Proposition~\ref{balasform}}

Note that the constraints define a convex set because the conic hull of a convex set is convex, and $||x_i||_2^2 \le z_it$ is a form of the rotated second-order cone, which is also convex.
Any feasible assignment of the integer vector $z$ has at most one nonzero component. Without loss of generality we may take this to be the first component, so $z_1 = 1$. Since $t$ is unrestricted in the positive direction, the constraint $||x_1||_2^2 \le t$ imposes no restrictions on the vector $x_1$ and $x_1 \in T_1$ iff there exists $y_1 \in \mathbb{R}^p$ such that $(x_1,y_1,1) \in \hat T_1$. For $i > 1$, the constraint $||x_i||_2^2 \le 0$ implies $x_i = 0$, and this is feasible because $(0,0,0) \in \hat T_i$ given $\hat T_i$ is nonempty by assumption. 

\subsection{Proof of Proposition~\ref{boundedprop}}
This argument is an extension of Theorem 11.6 of~\cite{VielmaSurvey}.
Suppose $M = C + K$ where $C$ is a compact convex set and $K$ is a polyhedral cone generated by rational rays $(r_x^1,r_y^1,r_z^1),(r_x^2,r_y^2,r_z^2),\ldots,(r_x^t,r_y^t,r_z^t)$. (We may assume without loss of generality that these rays furthermore have integer components). We will prove that there exist sets $S^1,\ldots,S^k$ such that
\begin{equation}\label{eq:finiteS}
S=\proj_x\bra{M\cap \bra{\mathbb{R}^{n+p} \times \mathbb{Z}^d}} = \cup_{i=1}^k S^{i} + \operatorname{intcone}(r_x^1,r_x^2,\ldots,r_x^t),
\end{equation}
where each $S^i$ is a projection of a closed convex set.

For any $(x^*,y^*,z^*) \in M$ there exist a finite (via Carath\'{e}odory) set of extreme points $(x^1,y^1,z^1),(x^2,y^2,z^2),\ldots,(x^w,y^w,z^w)$ of $C$ and nonnegative multipliers $\lambda$, $\gamma$ with $\sum_{i=1}^w \lambda_i = 1$ such that
\begin{equation}
(x^*,y^*,z^*) = \sum_{i=1}^w \lambda_i (x^i,y^i,z^i) + \sum_{j=1}^t \gamma_j (r_x^j,r_y^j,r_z^j).
\end{equation}
Define
\begin{equation}
(\hat x, \hat y, \hat z) = \sum_{i=1}^w \lambda_i (x^i,y^i,z^i) + \sum_{j=1}^t (\gamma_j - \lfloor \gamma_j \rfloor) (r_x^j,r_y^j,r_z^j)
\end{equation}
and
\begin{equation}
(x^\infty, y^\infty, z^\infty) = \sum_{j=1}^t \lfloor \gamma_j \rfloor (r_x^j,r_y^j,r_z^j)
\end{equation}
so that $(x^*,y^*,z^*) = (\hat x, \hat y, \hat z) + (x^\infty, y^\infty, z^\infty)$.

      Note that $(x^\infty,y^\infty,z^\infty) \in \operatorname{intcone}( (r_x^{1}, r_y^1, r_z^{1}), \ldots, (r_x^{t},r_y^t,r_z^{t})) =: M^\infty$ and $\hat z = z^* - z^\infty \in \mathbb{Z}^d$, so $(\hat x, \hat y, \hat z)$ belongs to a bounded set
      \begin{equation}
          \hat M = (C + B) \cap (\mathbb{R}^{n+p} \times \mathbb{Z}^d),
      \end{equation}
      where $B = \{ \sum_{j=1}^t \gamma_j (r_x^j,r_y^j,r_z^j) | 0 \le \gamma \le 1 \}$.
          Since this decomposition holds for any points $(x^*,y^*,z^*) \in M$, it follows that $M \subseteq \hat M +  M^\infty$.
          Since $\hat M$ is bounded, $\hat M$ is a finite union of bounded convex sets. Also $\hat M + M^\infty \subseteq M$ is easy to show, so we've demonstrated $M = \hat M + M^\infty$. The statement~\eqref{eq:finiteS} follows from projection of $M$ onto the $x$ variables.

\subsection{Proof of Lemma~\ref{MILPlemma}}

\begin{proof}
Assume $S$ is rational MICP representable. Then there exists finitely many points $b_1,\ldots,b_n$ and $z_1,\ldots,z_r$ such that $S=\{b_1,\ldots,b_n\}+\mathrm{intcone}(z_1,\ldots,z_r)$. If $r = 0$, then $S$ is a finite union of points and we are done. So we can assume $r>0$ and without loss of generality that $n=1$. Note that $\mathrm{intcone}(z_1,\ldots,z_r)=g \cdot I$ for $g:=\mathrm{gcd}(z_1,\ldots,z_n)$ and $I \subset  \mathbb{N}$ which is closed under addition and generated by finitely many coprime numbers. By Schur's Theorem~\cite{Bressoud}, $I$ consists of all eventually the non-negative natural numbers, so suppose $I:= \{\alpha_1,\ldots,\alpha_m \} \bigcup \{ x \in \mathbb{N} | x \geq \alpha_m \}$ for some $\alpha_1<\alpha_2<\ldots <\alpha_m \in \mathbb{N}$. So $S$ takes the form 
\begin{equation}\label{eq:Schur}
S= \{b_1+g\alpha_1,b_1+g \alpha_2, \ldots,b_1+g\alpha_m \} \bigcup \{ b_1+gx \in \mathbb{N} | x \geq \alpha_m \}.
\end{equation} But then if $J$ is the finite set defined by $J:=\{ y \leq 2 \prod_{i=1}^{m} \alpha_i | b_1+g y \in S\}$ we claim that $S=\bigcup_{ y \in J} \{ b_1+gy + (g \prod_{i=1}^{m} \alpha_i) K| K \in \mathbb{N} \}$.

It is easy to see that $\bigcup_{ y \in J} \{ b_1+gy + (g \prod_{i=1}^{m} \alpha_i) K| K \in \mathbb{N} \} \subseteq S$.For the other inclusion, take any $x$ such that $b_1+gx \in S$. Take the largest $M \in \mathbb{N}=\{0,1,2,\ldots \}$ such that $b_1+g(x-M \prod_{i=1}^m \alpha_i) \in S$. $M< \infty$ as $S \subseteq \mathbb{N}$. Now we claim $x-M \prod_{i=1}^m \alpha_i \in J$. Indeed, we have $b_1+g(x-M \prod_{i=1}^m \alpha_i )\in S$ and if $x-M \prod_{i=1}^m \alpha_i>2 \prod_{i=1}^m \alpha_i$ then $x-(M+1)\prod_{i=1}^m \alpha_i>\prod_{i=1}^m \alpha_i>\alpha_m$ and therefore $b_1+g(x-(M+1)\prod_{i=1}^m \alpha_i) \in S$ by (\ref{eq:Schur}) a contradiction with the definition of $M$. Hence, for some $y \in J$, $x= \prod_{i=1}^m \alpha_i M +y$ or $$b_1+gx=b_1+gy+g \prod_{i=1}^m \alpha_i M \in \bigcup_{ y \in J} \{ b_1+gy + (g \prod_{i=1}^{m} \alpha_i) K| K \in \mathbb{N} \}$$ as wanted.

The other direction is immediate.
\end{proof}
\subsection{Proof of Lemma~\ref{lem:finiteunion}}

\begin{proof}
Let $a_i,b_i \in \mathbb{N}$, $i \in \sidx{m}$ and suppose $T:=\bigcup_{i=1}^{m} \{a_i+b_ix|x \in \mathbb{N} \}$. Then to prove the Lemma we will prove that that for the finite set $J$ defined by $J:=\{ y \in \mathbb{N}, y \leq 2( \prod_{i=1}^m b_i+ \prod_{i=1}^{m} a_i )  | y \in T \}$ it holds $$T = \bigcup_{y \in J} \{ y+(\prod_{i=1}^{m} b_i)K | K \in \mathbb{N} \}.$$

It is easy to see that $\bigcup_{y \in J} \{ y+(\prod_{i=1}^{m} b_i)K | K \in \mathbb{N} \} \subseteq T$.

For the other inclusion, take any $i \in \sidx{m}$ and $x \in \mathbb{N}$. We need to show $a_i+b_i x \in \bigcup_{y \in J} \{ y+(\prod_{i=1}^{m} b_i)K | K \in \mathbb{N} \}$. Wlog $i=1$. Take the maximum $M \in \mathbb{N}$ such that $a_1+b_1(x-M \prod_{i=2}^{m}b_i) \in T$. We claim $a_1+b_1(x-M \prod_{i=2}^{m}b_i) \in J$. Indeed, if $a_1+b_1(x-M \prod_{i=2}^{m}b_i)>2( \prod_{i=1}^m b_i+ \prod_{i=1}^{m} a_i )$ then $a_1+b_1(x-(M+1) \prod_{i=2}^{m}b_i)>( \prod_{i=1}^m b_i+ \prod_{i=1}^{m} a_i )>a_1$ which in particular implies that $x-(M+1)\prod_{i=2}^{m}b_i >0$ and hence $a_1+b_1(x-(M+1)\prod_{i=2}^{m}b_i) \in T$, a contradiction with the definition of $T$. Hence $y:=a_1+b_1(x-M \prod_{i=2}^{m}b_i) \in J$ and hence for some $y \in J$, $a_1+b_1 x=y+M \prod_{i=1}^m b_i$ as needed.
\end{proof}

\subsection{Proof of Lemma~\ref{linearlemma}}

After an affine transformation we may assume without loss of generality that $\aff\bra{\set{x^i}_{i=1}^k}=\mathbb{R}^d$. Let  $\bar{h}:\mathbb{R}^d\to \mathbb{R}\cap\set{\infty}$ so that $\bar{h}(x)=h(x)$ for all $x\in  C$ and $\bar{h}(x)=\infty$ otherwise. We have that $x^1\in \Int\bra{\dom\bra{\bar{h}}}=\Int\bra{C}$ and hence $\partial \bar{h}\bra{x^1}$ is nonempty and bounded~\cite{rockafellar1997convex}. If there exist $u\in \partial \bar{h}\bra{x^1}\setminus \set{0}$ then for sufficiently small $\varepsilon>0$ we have $x^1+\varepsilon u\in \Int\bra{C}$ and $0\geq \bar{h}\bra{x^1+\varepsilon u}\geq \bar{h}\bra{x^1}+\varepsilon ||u||_2>0$, which is a contradiction. Hence $ \partial \bar{h}\bra{x^1}=\set{0}$ so $h\bra{x}=\bar{h}\bra{x}\geq  \bar{h}\bra{x^1}=0$ for all $x\in C$.

\end{document}